\documentclass[12pt,reqno]{amsart}
\usepackage{amsmath,amssymb,amsfonts,amscd,latexsym,amsthm,mathrsfs}
\usepackage[unicode]{hyperref}
\numberwithin{equation}{section}
\newtheorem{theorem}{Theorem}

\theoremstyle{remark}
\newtheorem*{remark}{\bf Remark}

\newtheorem{problem}{\bf Problem}
\newtheorem{conjecture}{\bf Conjecture}
\begin{document}

\title{O\protect\lowercase{n a $q$-deformation of modular forms}}

\author{V\protect\lowercase{ictor} J. W. G\protect\lowercase{uo} \& W\protect\lowercase{adim} Z\protect\lowercase{udilin}}
\address{School of Mathematical Sciences, Huaiyin Normal University, Huai'an 223300, Jiangsu, People's Republic of China}
\email{jwguo@hytc.edu.cn}

\address{Department of Mathematics, IMAPP, Radboud University, PO Box 9010, 6500~GL Nijmegen, Netherlands}
\email{w.zudilin@math.ru.nl}

\date{28 December 2018}

\begin{abstract}
There are many instances known when the Fourier coefficients of modular forms are congruent to partial sums of hypergeometric series.
In our previous work, such partial sums are related to the radial asymptotics of infinite $q$-hypergeometric sums at roots of unity.
Here we combine the two features to construct a hypergeometric $q$-deformation of two CM modular forms of weight~3
and discuss the corresponding $q$-congruences.
\end{abstract}

\thanks{The first author was partially supported by the National Natural Science Foundation of China (grant 11771175).}


\maketitle

\rightline{\em To Bruce Berndt, with admiration and warm wishes,}
\rightline{\em on his $\dfrac{(q;q)_6}{(1-q)^4(1-q^3)^2}$ birthday, as $q\to1$}

\bigskip
\section{Introduction}
\label{sec-intro}
The hypergeometric identity
\begin{equation}
\sum_{k=0}^\infty(4k+1)\frac{(\frac12)_k^3}{k!^3}\,(-1)^k=\frac2\pi,
\label{eq-HB}
\end{equation}
due to G.~Bauer \cite{Ba59}, is more than 150 years old but still attracts a lot of mathematical interest
because of its belonging to a family of the so-called Ramanujan-type identities for $1/\pi$ (see \cite{BBC09,EZ94,Le10,Ra14,Zu08} and, in particular, \cite[Chap.~15]{AB09} and \cite[Chap.~14]{Co17}).
Here $(a)_k=\prod_{j=0}^{k-1}(a+j)=\Gamma(a+k)/\Gamma(a)$ is the Pochhammer notation, so that $(\frac12)_k/k!=\binom{2k}k2^{-2k}$ for $k=0,1,2,\dots$\,.
One arithmetic manifestation of a special status of \eqref{eq-HB} are supercongruences
\begin{align}
\sum_{k=0}^{p-1}(4k+1)\frac{(\frac12)_k^3}{k!^3}\,(-1)^k
&\equiv\sum_{k=0}^{(p-1)/2}(4k+1)\frac{(\frac12)_k^3}{k!^3}\,(-1)^k
\nonumber\\
&\equiv(-1)^{(p-1)/2}p\pmod{p^3}
\quad\text{for primes}\; p>2,
\label{eq-cong0}
\end{align}
observed by L.~Van Hamme \cite[Conjecture (B.2)]{VH97} and proved subsequently by E.~Mortenson \cite{Mo08} (see also \cite{Zu09}).

A principal objective of our earlier work \cite{GZ18b} was demonstration that the hypergeometric evaluations like \eqref{eq-HB}
and congruences for truncated sums like \eqref{eq-cong0} may be deduced, in a uniform way, from suitable $q$-deformations of \eqref{eq-HB}.
Namely, the asymptotics of such $q$-deformations as $q$ tends radially to a root of unity governs the behaviour of partial sums related to the degree of that root.
Notice that the right-hand sides $a(p)=(-1)^{(p-1)/2}p$ of the congruences \eqref{eq-cong0}
combine to the Dirichlet generating function
\begin{equation}
L(s)=\prod_{p>2}(1-a(p)p^{-s})^{-1}
=\sum_{m=0}^\infty\frac{(-1)^m}{(2m+1)^{s-1}}=L(\chi_{-4},s-1),
\label{eq-L}
\end{equation}
where $\chi_{-4}=\bigl(\frac{-1}{\cdot}\bigr)$ is the nonprincipal modulo 4 character,
and the series evaluates to $\pi/4$ at $s=2$, so that \eqref{eq-HB} transforms into
\begin{equation*}
\sum_{k=0}^\infty(4k+1)\frac{(\frac12)_k^3}{k!^3}\,(-1)^k=\frac{8L(2)}{\pi^2}.
\end{equation*}
For some other recent progress on $q$-congruences, we refer the reader to \cite{Guo19a,Guo19b,Guo19c,Guo18a,GS19,GS18}.

A theme of this note is to give examples of $q$-deformations of hypergeometric evaluations,
which are linked with the coefficients of modular forms rather than Dirichlet characters,
and use these $q$-hypergeometric identities to establish the corresponding (super)congruences.

\section{Hypergeometric identities and congruences}
\label{sec-hyper}

A forward player of our exposition is the hypergeometric evaluation
\begin{equation*}
\sum_{k=0}^\infty\frac{(\frac12)_k^3}{k!^3}
=\frac{\pi}{\Gamma(3/4)^4}.
\end{equation*}
Its right-hand side happens to be a (simple multiple of the) period of the CM modular form
\begin{equation}
f_1(\tau)=q\prod_{m=1}^\infty(1-q^{4m})^6=\sum_{n=1}^\infty a_1(n)q^n,
\quad\text{where}\; q=\exp(2\pi i\tau),
\label{eq-f1}
\end{equation}
of weight~3:
\begin{equation}
\sum_{k=0}^\infty\frac{(\frac12)_k^3}{k!^3}
=\frac{\Gamma(1/2)^2}{\Gamma(3/4)^4}
=\frac{16L(f_1,2)}{\pi^2}
=\frac{8L(f_1,1)}{\pi}
\label{eq-H1}
\end{equation}
(see \cite[Theorem~5]{RWZ15}), where $L(f_1,s)$ denotes the Dirichlet $L$-function of~\eqref{eq-f1}.
This relationship between the hypergeometric series and modular form
is somewhat deeper because of the chain of related congruences
\begin{equation}
\sum_{k=0}^{p-1}\frac{(\frac12)_k^3}{k!^3}
\equiv a_1(p)\pmod{p^2}
\quad\text{for primes}\; p>2
\label{eq-cong1}
\end{equation}
(see \cite{On98,ZHS13}), which link corresponding partial sums with the Fourier coefficients of \eqref{eq-f1}.
The latter can be given explicitly via
\begin{equation}
a_1(p)
=\begin{cases}
2(a^2-b^2) & \text{if $p=a^2+b^2$, $a$ odd}, \\
0 & \text{if $p\equiv3\pmod4$},
\end{cases}
\label{eq-a1}
\end{equation}
and, in turn, satisfy
\begin{equation*}
a_1(p)\equiv\frac{\Gamma_p(1/2)^2}{\Gamma_p(3/4)^4}\pmod{p^2}
\quad\text{for}\; p>2
\end{equation*}
(see \cite[Sect.~1 and Conjecture (A.2)]{VH97}). This visually makes the series in \eqref{eq-f1} a plausible
generating function of the truncations in~\eqref{eq-cong1}.
Furthermore, we remark that \cite{ZHS13}
\begin{equation*}
a_1(p)\equiv\frac{(\frac12)_m^2}{m!^2}\pmod{p^2}
\quad\text{if}\; m=(p-1)/4\in\mathbb Z.
\end{equation*}

We should stress that not every formal $q$-analogue of a hypergeometric summation, $|q|<1$, may suit for application of the general machinery from our earlier work \cite{GZ18b}.
An example of $q$-deforming \eqref{eq-H1} is given by
\begin{equation}
\sum_{k=0}^\infty(1+q^{4k+1})\frac{(q^2;q^4)_k^3}{(q^4;q^4)_k^3}\,q^k
=\frac{(q^2;q^4)_\infty^2(q^3;q^4)_\infty^2}{(q;q^4)_\infty^2(q^4;q^4)_\infty^2}
\label{eq-HqA}
\end{equation}
(apply \cite[eq.~(52)]{GZ18b} with $a=b=c=1$),
where the $q$-notation $(a;q)_k$ stands for $\prod_{j=0}^{k-1}(1-aq^j)$ for $k=0,1,2,\dots,\infty$.
Identity \eqref{eq-HqA} originates however from
\begin{align}
\sum_{k=0}^\infty\frac{1-q^{4k+1}}{1-q}\,\frac{(q^2;q^4)_k^3}{(q^4;q^4)_k^3}\,(-q)^k
&=\frac{(q^2;q^4)_\infty^2(-q^3;q^4)_\infty^2}{(1-q)\,(-q;q^4)_\infty^2(q^4;q^4)_\infty^2}
\nonumber\\
&=\frac{(q^5;q^4)_\infty^4 (q^6;q^8)_\infty^4}{(q;q^2)_\infty^2 (q^4;q^4)_\infty^2}
\label{eq-HqB}
\end{align}
(replace $q$ with $-q$), which makes it a $q$-analogue of Bauer's formula \eqref{eq-HB}.
This `true' origin \eqref{eq-HB} makes the asymptotics of \eqref{eq-HqA} at roots of unity $q$ related to the truncated sums of \eqref{eq-HB} rather than \eqref{eq-H1}.

In this note, we give a $q$-extension of \eqref{eq-H1} that accommodates the related congruences \eqref{eq-cong1} and, by these means, implicitly provides a $q$-generalisation of the generating function \eqref{eq-f1}.
We also provide a similar $q$-extension of the formula
\begin{align}
{}_3F_2\biggl(\begin{matrix} \frac12, \, \frac12, \, \frac12 \\ 1, \, 1 \end{matrix}\biggm|-1\biggr)
=\frac{\Gamma(1/2)^2}{\sqrt2\,\Gamma(5/8)^2\Gamma(7/8)^2}
=\frac{12\sqrt2\,L(f_2,2)}{\pi^2}
=\frac{12L(f_2,1)}{\pi}
\label{eq-H2}
\end{align}
(see \cite[Theorem~5]{RWZ15}), which underlies \eqref{eq-HB} and is attached to the weight~3 CM modular form
\begin{equation*}
f_2(\tau)=\sum_{n=1}^\infty a_2(n)q^n
=q\prod_{m=1}^\infty(1-q^m)^2(1-q^{2m})(1-q^{4m})(1-q^{8m})^2.
\end{equation*}
In this case, we have \cite{On98}
\begin{equation*}
a_2(p)=\begin{cases}
2(a^2-2b^2) & \text{if $p=a^2+2b^2$}, \\
0 & \text{if $p\equiv5,7\pmod8$},
\end{cases}
\end{equation*}
as well as
\begin{equation}
(-1)^{(p-1)/2}\sum_{k=0}^{p-1}\frac{(\frac12)_k^3}{k!^3}(-1)^k
\equiv\frac{\Gamma_p(1/2)^2}{\Gamma_p(5/8)^2\Gamma_p(7/8)^2}
\equiv a_2(p)\pmod{p^2}
\label{eq-cong2}
\end{equation}
for primes $p>2$ and
\begin{equation*}
a_2(p)\equiv\frac{(\frac14)_m^2}{m!^2}\times\begin{cases}
\phantom-1 &\text{if}\; m=(p-1)/8\in\mathbb Z \\
-\frac12 &\text{if}\; m=(3p-1)/8\in\mathbb Z
\end{cases} \pmod{p^2}.
\end{equation*}

\section{A $q$-Clausen identity}
\label{sec-Clausen}

F.\,H.~Jackson's generalisation of Clausen's identity \cite{Ja41} (see also \cite[eq.~(3.2)]{Sch18}) implies
\begin{align}
{}_2\phi_1\biggl[\begin{matrix} qa, \, q/a \\ q^4 \end{matrix}; q^4,z\biggr]
\,{}_2\phi_1\biggl[\begin{matrix} qa, \, q/a \\ q^4 \end{matrix};q^4,q^2z\biggr]
=\sum_{k=0}^\infty\frac{(aq;q^2)_k(q/a;q^2)_k(q^2;q^4)_k}{(q^2;q^2)_k^2(q^4;q^4)_k}\,z^k,
\label{eq:jackson}
\end{align}
where the \emph{basic hypergeometric series} is defined as
$$
{}_{r+1}\phi_{r}\biggl[\begin{matrix}
a_1, \, a_2, \, \dots, \, a_{r+1} \\ b_1, \, b_2, \, \dots, \, b_{r}
\end{matrix}; q,\, z \biggr]
=\sum_{k=0}^{\infty}\frac{(a_1;q)_k(a_2;q)_k\dotsb(a_{r+1};q)_k z^k}
{(q;q)_k(b_1;q)_k(b_2;q)_k\dotsb(b_{r};q)_k}.
$$
Clearly, if one takes $a=1$ and $z=q^2$ in \eqref{eq:jackson}, then the first series on the left is $q$-Gauss-summable,
$$
{}_2\phi_1\biggl[\begin{matrix} q, \, q \\ q^4 \end{matrix}; q^4,q^2\biggr]
=\frac{(q^3,q^3;q^4)_\infty}{(q^2,q^4;q^4)_\infty},
$$
so that this leads to a `natural' reduction of the series on the right and makes no surprise from divisibility
of the result by some cyclotomic polynomials
$$
\Phi_n(q)=\prod_{\substack{j=1\\ \gcd(j,n)=1}}^n(q-e^{2\pi ij/n}).
$$

\begin{theorem}
\label{th1}
For any positive integer $n$, we have
\begin{align}
\sum_{k=0}^{n-1}\frac{(q;q^4)_k^2}{(q^4;q^4)_k^2}\,q^{2k}
\equiv
\begin{cases}
\dfrac{(q^{3};q^4)_{(n-1)/4}}{(q^4;q^4)_{(n-1)/4}}\pmod{\Phi_n(q)} &\text{if $n\equiv 1\pmod{4}$}, \\[3pt]
0 \pmod{\Phi_n(q)} &\text{if $n\equiv 3\pmod{4}$}.
\end{cases}
\label{eq:114-1}
\end{align}
\end{theorem}

\begin{proof}
From the proof of \cite[Theorem 4.1]{GS18} we know that
\begin{align*}
(q;q^4)_k^2\equiv (q^{1-3n};q^4)_k (q^{1+3n};q^4)_k \pmod{\Phi_n(q)^2}.
\end{align*}
Therefore, for $n\equiv 3\pmod{4}$ we have, modulo $\Phi_n(q)^2$,
\begin{align*}
\sum_{k=0}^{n-1}\frac{(q;q^4)_k^2}{(q^4;q^4)_k^2}\,q^{2k}
\equiv \sum_{k=0}^{(3n-1)/4}\frac{(q^{1-3n};q^4)_k (q^{1+3n};q^4)_k}{(q^4;q^4)_k^2}\,q^{2k}
=\frac{(q^{3-3n};q^4)_{(3n-1)/4}}{(q^4;q^4)_{(3n-1)/4}}
\end{align*}
by the $q$-Chu--Vandermonde summation formula \cite[Appendix (II.7)]{GR04}:
\begin{align}
{}_{2}\phi_{1}\biggl[\begin{matrix} a, \, q^{-n} \\ c \end{matrix}; q, \, \frac{cq^n}{a}\biggr]
=\frac{(c/a;q)_n}{(c;q)_n}.
\label{eq:q-Chu}
\end{align}
The proof then follows from the fact that $(q^{3-3n};q^4)_{(3n-1)/4}$ contains the factor $1-\nobreak q^{-2n}$.

Similarly, for $n\equiv 1\pmod{4}$, the result follows from
\begin{equation*}
(q;q^4)_k^2\equiv (q^{1-n};q^4)_k (q^{1+n};q^4)_k \pmod{\Phi_n(q)^2}.
\qedhere
\end{equation*}
\end{proof}

Making the substitution $q\to q^{-1}$ in \eqref{eq:114-1}, we obtain
\begin{align*}
\sum_{k=0}^{n-1}\frac{(q;q^4)_k^2}{(q^4;q^4)_k^2}\,q^{4k}
\equiv\begin{cases}
q^{(n^2-1)/4}\dfrac{(q^{3};q^4)_{(n-1)/4}}{(q^4;q^4)_{(n-1)/4}}\pmod{\Phi_n(q)} &\text{if $n\equiv 1\pmod{4}$}, \\
0 \pmod{\Phi_n(q)} &\text{if $n\equiv 3\pmod{4}$}.
\end{cases}
\end{align*}

Remarkably, the truncations
$$
\sum_{k=0}^{n-1}\frac{(q;q^2)_k^2(q^2;q^4)_k}{(q^2;q^2)_k^2(q^4;q^4)_k}\,q^{2k}
$$
of the series on the right-hand side in \eqref{eq:jackson} when $a=1$ and $z=q^2$
are divisible by more cyclotomic polynomials, always squared, than the corresponding sums in \eqref{eq:114-1}
for all $n\ge2$; the congruence
$$
\sum_{k=0}^{n-1}\frac{(q;q^2)_k^2(q^2;q^4)_k}{(q^2;q^2)_k^2(q^4;q^4)_k}\,q^{2k}
\equiv0\pmod{\Phi_n(q)^2}
$$
for $n\equiv3\pmod4$, not necessarily prime, discussed in \cite[Corollary 1.2]{GZ15} is a particular instance of this high divisibility.
We numerically observe that
$$
\sum_{k=0}^{n-1}\frac{(q;q^2)_k^2(q^2;q^4)_k}{(q^2;q^2)_k^2(q^4;q^4)_k}\,q^{2k}
\equiv\begin{cases}
q^{(n-1)/2}\dfrac{(q^2;q^2)_{(n-1)/2}^2}{(q^4;q^4)_{(n-1)/4}^4} &\text{if}\; n\equiv1\pmod4, \\[2.5pt]
0 &\text{if}\; n\equiv3\pmod4,
\end{cases}
$$
is true not just modulo $\Phi_n(q)^2$ but also modulo $\prod_{\ell\equiv3\pmod4,\,\ell\mid n}\Phi_\ell(q)^2$ for $n\equiv1\pmod4$ and even modulo $\prod_{\ell\equiv3\pmod4,\,\ell<n}\Phi_\ell(q)^2$ if $n\equiv3\pmod4$.

\begin{theorem}
\label{thm:moreover}
Modulo $\Phi_n(q)^2$,
\begin{align}
\sum_{k=0}^{n-1}\frac{(q;q^2)_k^2(q^2;q^4)_k}{(q^2;q^2)_k^2(q^4;q^4)_k}\,q^{2k}
\equiv\begin{cases}
q^{(n-1)/2}\dfrac{(q^2;q^4)_{(n-1)/4}^2}{(q^4;q^4)_{(n-1)/4}^2} &\text{if}\; n\equiv1\pmod4, \\[2.5pt]
0 &\text{if}\; n\equiv3\pmod4.
\end{cases}
\label{eq:mod-phi}
\end{align}
\end{theorem}

\begin{proof}
Although the method of proving \cite[Corollary 1.2]{GZ15} can also be used to establish \eqref{eq:mod-phi}, here we give
a somewhat different argument.
We shall demonstrate a parametric generalisation of \eqref{eq:mod-phi}, namely, that
\begin{align}
\sum_{k=0}^{n-1}\frac{(aq;q^2)_k (q/a;q^2)_k (q^2;q^4)_k}{(q^2;q^2)_k^2(q^4;q^4)_k}\,q^{2k}
\equiv\begin{cases}
q^{(n-1)/2}\dfrac{(q^2;q^4)_{(n-1)/4}^2}{(q^4;q^4)_{(n-1)/4}^2} &\text{if}\; n\equiv1\pmod4, \\[2.5pt]
0 &\text{if}\; n\equiv3\pmod4,
\end{cases}
\label{eq:with-a}
\end{align}
holds true modulo $(1-aq^n)(a-q^n)$.
For $a=q^{-n}$ or $a=q^n$, the left-hand side of \eqref{eq:with-a} is equal to
\begin{align}
\sum_{k=0}^{(n-1)/2}\frac{(q^{1-n};q^2)_k (q^{1+n};q^2)_k (q^2;q^4)_k}{(q^2;q^2)_k^2(q^4;q^4)_k}\,q^{2k}
={}_{4}\phi_{3}\biggl[\begin{matrix} q^{1-n}, \, q^{n+1}, \, q, \, -q \\
q^2,\, -q^2,\, q^2\end{matrix}; q^2, \, q^2\biggr].
\label{eq:with-a-qn}
\end{align}
By Andrews' terminating $q$-analogue of Watson's formula (see \cite{An76} or \cite[Appendix (II.17)]{GR04}),
\begin{align*} 
{}_{4}\phi_{3}\biggl[\begin{matrix} q^{-m}, \, a^{2}q^{m+1}, \, b, \, -b \\
aq,\, -aq,\, b^{2}\end{matrix}; q, \, q\biggr]
=\begin{cases}
\dfrac{b^{m}(q, a^{2}q^{2}/b^{2}; q^{2})_{m/2}}{(a^{2}q^{2},\, b^{2}q; q^{2})_{m/2}} &\text{if $m$ is even}, \\[2.5pt]
0 &\text{if $m$ is odd},
\end{cases}
\end{align*}
we conclude that the right-hand side of \eqref{eq:with-a-qn} is just that of \eqref{eq:with-a}.
Finally, letting $a\to1$ in \eqref{eq:with-a}, we are led to \eqref{eq:mod-phi}.
\end{proof}

\begin{remark}
The $n\equiv 3\pmod{4}$ case can also be deduced from  \cite[Theorem 1.1]{Guo18a}. For $n\equiv 1\pmod{4}$,
we can also prove \eqref{eq:mod-phi} without using Andrews' formula as follows.
By \eqref{eq:jackson} and \eqref{eq:q-Chu}, we have, modulo $\Phi_n(q)^2$,
\begin{align*}
\sum_{k=0}^{n-1}\frac{(q;q^2)_k^2(q^2;q^4)_k}{(q^2;q^2)_k^2(q^4;q^4)_k}\,q^{2k}
&\equiv\sum_{k=0}^{n-1}\frac{(q^{1-n};q^2)_k (q^{1+n};q^2)_k(q^2;q^4)_k}{(q^2;q^2)_k^2(q^4;q^4)_k}\,q^{2k}
\\
&=\sum_{k=0}^{n-1}\frac{(q^{1-n};q^4)_k (q^{1+n};q^4)_k}{(q^4;q^4)_k^2}\,q^{2k}
\\ &\qquad\times
\sum_{k=0}^{n-1}\frac{(q^{1-n};q^4)_k (q^{1+n};q^4)_k}{(q^4;q^4)_k^2}\,q^{4k}
\\
&=q^{(n^2-1)/4}\frac{(q^{3-n};q^4)_{(n-1)/4}^2}{(q^4;q^4)_{(n-1)/4}^2},
\end{align*}
which is in fact the same as \eqref{eq:mod-phi}.
\end{remark}

The consideration above corresponds to a $q$-deformation of \eqref{eq-cong1};
our $q$-analogue of \eqref{eq-cong2} is somewhat similar but weaker.

We first prove the following result.

\begin{theorem}\label{thm:57-1}
For any positive odd integer $n$, we have, modulo $\Phi_n(q)$,
\begin{align}
\sum_{k=0}^{n-1}\frac{(q;q^4)_k^2 } {(q^4;q^4)_k^2}(-q)^{3k}
\equiv\begin{cases}
\dfrac{(q^5,q^7;q^8)_{(n-1)/8} }{(q^4;q^4)_{(n-1)/4}} &\text{if $n\equiv 1\pmod{8}$}, \\[2.5pt]
\dfrac{(q^5,q^7;q^8)_{(3n-1)/8} } {(q^4;q^4)_{(3n-1)/4}} &\text{if $n\equiv 3\pmod{8}$}, \\[2.5pt]
0,&\text{if $n\equiv 5,7\pmod{8}$}.
\end{cases}
\label{eq:8k-1357}
\end{align}
\end{theorem}

\begin{proof}
We use the $q$-Kummer (Bailey--Daum) summation formula \cite[Appendix (II.9)]{GR04}:
\begin{align}
 {}_{2}\phi_{1}\biggl[\begin{matrix} a, \, b \\ aq/b \end{matrix}; q, \, \frac{-q}{b}\biggr]
=\frac{(-q;q)_\infty(aq;q^2)_\infty (aq^2/b^2;q^2)_\infty}{(-q/b;q)_\infty (aq/b;q)_\infty}.
\label{eq:qkummer}
\end{align}
For $n\equiv 1\pmod{8}$, by \eqref{eq:qkummer} we obtain
\begin{align}
{}_{2}\phi_{1}\biggl[\begin{matrix} q^{1-n}, \, q \\ q^{4-n}\end{matrix}; q^4, \, -q^3\biggr]
&=\frac{(-q^4;q^4)_\infty(q^{5-n};q^8)_\infty (q^{7-n};q^8)_\infty}{(-q^3;q^4)_\infty (q^{4-n};q^4)_\infty}
\nonumber\\
&=\frac{(q^{5-n};q^8)_{(n-1)/8} (q^{7-n};q^8)_{(n-1)/8}}{(q^{4-n};q^4)_{(n-1)/4}}.
\label{eq:8k-1}
\end{align}
Similarly, for $n\equiv 3\pmod{8}$, we have
\begin{align}
{}_{2}\phi_{1}\biggl[\begin{matrix} q^{1-3n}, \, q \\ q^{4-3n}\end{matrix}; q^4, \, -q^3\biggr]
=\frac{(q^{5-3n};q^8)_{(3n-1)/8} (q^{7-3n};q^8)_{(3n-1)/8}}{(q^{4-3n};q^4)_{(3n-1)/4}}.
\label{eq:8k-3}
\end{align}
Furthermore,
\begin{align}
\sum_{k=0}^{n-1}\frac{(q^{1-n};q^4)_k (q;q^4)_k} {(q^4;q^4)_k (q^{4-n};q^4)_k}(-q)^{3k}&=0
\quad\text{for} \; n\equiv 5\pmod{8},
\label{eq:8k-5}
\\
\sum_{k=0}^{n-1}\frac{(q^{1-3n};q^4)_k (q;q^4)_k} {(q^4;q^4)_k (q^{4-3n};q^4)_k}(-q)^{3k}&=0
\quad\text{for} \; n\equiv 7\pmod{8}.
\label{eq:8k-7}
\end{align}
Since $q^{3n}\equiv q^n\equiv 1\pmod{\Phi_n(q)}$, we deduce the desired $q$-congruences \eqref{eq:8k-1357} from \eqref{eq:8k-1}--\eqref{eq:8k-7}.
\end{proof}

We complement Theorem \ref{thm:57-1} with the following related result.

\begin{theorem}\label{thm:57-2}
For any positive odd integer $n$, we have, modulo $\Phi_n(q)$,
\begin{align*}
&\sum_{k=0}^{n-1}\frac{(q;q^4)_k^2 } {(q^4;q^4)_k^2}(-q)^{k}
\\
&\quad\equiv\begin{cases}
\phantom-\dfrac{(q,q^2,-q^3,-q^4;q^4)_{(n-1)/8} } {(q^4;q^4)_{(n-1)/4}}q^{(1-n^2)/8} &\text{if $n\equiv 1\pmod{8}$},\\[7.5pt]
\phantom-\dfrac{(q,q^2,-q^3,-q^4;q^4)_{(3n-1)/8} } {(q^4;q^4)_{(3n-1)/4}}q^{(1-n^2)/8} &\text{if $n\equiv 3\pmod{8}$},\\[7.5pt]
-\dfrac{(1-q^2)(-q^3,-q^4,q^5,q^6;q^4)_{(n-5)/8}} {(q^4;q^4)_{(n-5)/4}}q^{(9-n^2)/8} &\text{if $n\equiv 5\pmod{8}$},\\[7.5pt]
-\dfrac{(1-q^2)(-q^3,-q^4,q^5,q^6;q^4)_{(3n-5)/8}} {(q^4;q^4)_{(3n-5)/4}}q^{(9-n^2)/8} &\text{if $n\equiv 7\pmod{8}$}.
\end{cases}
\end{align*}
\end{theorem}

\begin{proof}
Letting $q\to q^4$, $b=q$ and $c=-q^{2-4N}$ in Jackson's terminating $q$-analogue of Dixon's sum (see \cite[Appendix (II.15)]{GR04}),
\begin{align}
 {}_{3}\phi_{2}\biggl[\begin{matrix} q^{-2N}, \, b, \, c \\ q^{1-2N}/b,\, q^{1-2N}/c \end{matrix};
q, \, \frac{q^{2-N}}{bc}\biggr]
=\frac{(b,c;q)_N (q,bc;q)_{2N}}{(q,bc;q)_N (b,c;q)_{2N}},
\label{eq:qDxion}
\end{align}
we obtain
\begin{align*}
 {}_{2}\phi_{1}\biggl[\begin{matrix} q^{-8N}, \, q \\ q^{3-8N}\end{matrix}; q^4, \, -q^5\biggr]
=\frac{(q,-q^{2-4N};q^4)_N (q^4,-q^{3-4N};q^4)_{2N}}{(q^4,-q^{3-4N};q^4)_N (q,-q^{2-4N};q)_{2N}}.
\end{align*}
Taking $q\to q^{-1}$ in the above identity, we are led to
\begin{align*}
 {}_{2}\phi_{1}\biggl[\begin{matrix} q^{-8N}, \, q \\ q^{3-8N} \end{matrix}; q^4, \, -q\biggr]
=\frac{(q,-q^{2-4N};q^4)_N (q^4,-q^{3-4N};q^4)_{2N}\,q^{-4N}}
{(q^4,-q^{3-4N};q^4)_N (q,-q^{2-4N};q)_{2N}}.
\end{align*}
It follows that, for $n\equiv 1\pmod{8}$,
\begin{align*}
\sum_{k=0}^{n-1}\frac{(q;q^4)_k^2 } {(q^4;q^4)_k^2}(-q)^{k}
&\equiv \sum_{k=0}^{(n-1)/4}\frac{(q^{1-n};q^4)_k (q;q^4)_k } {(q^4;q^4)_k (q^{4-n};q^4)_k}(-q)^{k} \\[5pt]
&=\frac{(q,-q^{(5-n)/2};q^4)_{(n-1)/8} (q^4,-q^{(7-n)/2};q^4)_{(n-1)/4}q^{(1-n)/2}}
{(q^4,-q^{(7-n)/2};q^4)_{(n-1)/4}q^{(1-n)/8}(q,-q^{(5-n)/2};q^4)_{(n-1)/4} }  \\[5pt]
&=\frac{(q,q^2,-q^3,-q^4;q^4)_{(n-1)/8} } {(q^{4-n};q^4)_{(n-1)/4}}q^{(1-n^2)/8}.
\end{align*}
Since $q^n\equiv 1\pmod{\Phi_n(q)}$, we complete the proof of the theorem for the first case.
The other three cases follow in a similar way: modulo $\Phi_n(q)$,
\begin{alignat*}{2}
\sum_{k=0}^{n-1}\frac{(q;q^4)_k^2 } {(q^4;q^4)_k^2}(-q)^{k}
&\equiv \sum_{k=0}^{(3n-1)/4}\frac{(q^{1-3n};q^4)_k (q;q^4)_k } {(q^4;q^4)_k (q^{4-3n};q^4)_k}(-q)^{k}
&\quad &\text{for}\; n\equiv 3\pmod{8},
\\
\sum_{k=0}^{n-1}\frac{(q;q^4)_k^2 } {(q^4;q^4)_k^2}(-q)^{k}
&\equiv \sum_{k=0}^{(n-1)/4}\frac{(q^{1-n};q^4)_k (q;q^4)_k } {(q^4;q^4)_k (q^{4-n};q^4)_k}(-q)^{k}
&\quad &\text{for}\; n\equiv 5\pmod{8},
\\
\sum_{k=0}^{n-1}\frac{(q;q^4)_k^2 } {(q^4;q^4)_k^2}(-q)^{k}
&\equiv\sum_{k=0}^{(3n-1)/4}\frac{(q^{1-3n};q^4)_k (q;q^4)_k } {(q^4;q^4)_k (q^{4-3n};q^4)_k}(-q)^{k}
&\quad &\text{for}\; n\equiv 7\pmod{8},
\end{alignat*}
and for the last two instances, we use the `odd version' of Jackson's $q$-analogue of Dixon's sum
(which follows from \cite[eq.~(2.3)]{Ba50}),
\begin{align*}
 {}_{3}\phi_{2}\biggl[\begin{matrix} q^{1-2N}, \, b, \, c \\ q^{2-2N}/b, \, q^{1-2N}/c\end{matrix};
q, \, \frac{q^{2-N}}{bc}\biggr]
=\frac{(b,c;q)_N (q,bc;q)_{2N-1}}{(q;q)_{N-1}(bc;q)_N (b;q)_{2N-1}(c;q)_{2N}},
\end{align*}
instead of \eqref{eq:qDxion}.
\end{proof}

\begin{theorem}
\label{th5}
For any positive integer $n\equiv 1\pmod{4}$, modulo $\Phi_n(q)$, we have
\begin{align*}
&\sum_{k=0}^{(n-1)/2}\frac{(q;q^2)_k^2 (q^2;q^4)_k} {(q^2;q^2)_k^2 (q^4;q^4)_k}\,(-q)^k
\\
&\quad\equiv\begin{cases}
\dfrac{(q,q^2;q^4)_{(n-1)/8}^2 (-q;q)_{(n-1)/2} } {(q^4;q^4)_{(n-1)/4}^2}\,q^{-(1-n)^2/4} &\text{if $n\equiv 1\pmod{8}$},\\[7.5pt]
0 &\text{if $n\equiv 5\pmod{8}$}.
\end{cases}
\end{align*}
\end{theorem}

\begin{proof}
For $n\equiv 1\pmod{8}$, in Jackson's $q$-Clausen identity \eqref{eq:jackson}
we take $a=q^{-n}$ and $z=-q$ to obtain
\begin{align}
&\sum_{k=0}^{(n-1)/2}\frac{(q;q^2)_k^2(q^2;q^4)_k}{(q^2;q^2)_k^2(q^4;q^4)_k} (-q)^k
\notag\\
&\quad
\equiv\sum_{k=0}^{n-1}\frac{(q^{1-n};q^2)_k (q^{1+n};q^2)_k(q^2;q^4)_k}{(q^2;q^2)_k^2(q^4;q^4)_k}(-q)^{k}
\notag\\
&\quad
=\sum_{k=0}^{n-1}\frac{(q^{1-n};q^4)_k (q^{1+n};q^4)_k}{(q^4;q^4)_k^2}(-q)^k
\notag\\ &\quad\qquad\times
\sum_{k=0}^{n-1}\frac{(q^{1-n};q^4)_k (q^{1+n};q^4)_k}{(q^4;q^4)_k^2}(-q)^{3k}
\pmod{\Phi_n(q)^2}.
\label{eq:57-3}
\end{align}
By \eqref{eq:8k-1357}, we see that the second sum is congruent to
\begin{align*}
\frac{(q^5,q^7;q^8)_{(n-1)/8} }{(q^4;q^4)_{(n-1)/4}}
&\equiv \frac{(q^{5-n},q^{7-n};q^8)_{(n-1)/8} }{(q^4;q^4)_{(n-1)/4}}
\\
&=\frac{(q^2,q^4;q^8)_{(n-1)/8}q^{-(n-1)(n-3)/8} }{(q^4;q^4)_{(n-1)/4}} \pmod{\Phi_n(q)},
\end{align*}
while by Theorem \ref{thm:57-2} the first sum is congruent to
\begin{align*}
 {}_{2}\phi_{1}\biggl[\begin{matrix} q^{1-n}, \, q \\ q^{4-n} \end{matrix}; q^4, \, -q\biggr]
\equiv \frac{(q,q^2,-q^3,-q^4;q^4)_{(n-1)/8} } {(q^4;q^4)_{(n-1)/4}}q^{(1-n^2)/8}\pmod{\Phi_n(q)}.
\end{align*}
This establishes the $n\equiv 1\pmod{8}$ case of the theorem after some simplifications.

For $n\equiv 5\pmod{8}$, we again have \eqref{eq:57-3}. Since $q^n\equiv 1\pmod{\Phi_n(q)}$,
we know that the second sum on the right-hand side of \eqref{eq:57-3} is congruent to $0$ modulo $\Phi_n(q)$ by Theorem \ref{thm:57-1}, and so is
the right-hand side of \eqref{eq:57-3}. This proves the theorem for $n\equiv 5\pmod{8}$.
\end{proof}

We leave the related cases when $n\equiv3\pmod4$ of Theorem~\ref{th5} as an open problem to the reader.

\begin{problem}
\label{ec1}
For any positive integer $n\equiv 7\pmod{8}$, show that
\begin{align*}
\sum_{k=0}^{(n-1)/2}\frac{(q;q^2)_k^2 (q^2;q^4)_k} {(q^2;q^2)_k^2 (q^4;q^4)_k}\,(-q)^k
\equiv 0\pmod{\Phi_n(q)}.
\end{align*}
Give a related $q$-congruence for $n\equiv3\pmod8$.
\end{problem}

\section{Conclusion and open questions}
\label{sec-conc}

We have the following generalisation of Theorem~\ref{thm:moreover}
for $n\equiv 3\pmod{4}$, which (partly) forms the grounds of the arithmetic observations preceding the statement of the theorem.

\begin{conjecture}\label{conj:4k+3-2}
For $n\equiv 3\pmod{4}$ and any positive integer $r$, we have
\begin{align*}
\sum_{k=0}^{rn-1}\frac{(q;q^2)_k^2(q^2;q^4)_k}{(q^2;q^2)_k^2(q^4;q^4)_k}\,q^{2k}
&\equiv 0 \pmod{\Phi_n(q)^2},
\\
\sum_{k=0}^{rn+(n-1)/2}\frac{(q;q^2)_k^2(q^2;q^4)_k}{(q^2;q^2)_k^2(q^4;q^4)_k}\,q^{2k}
&\equiv 0 \pmod{\Phi_n(q)^2}.
\end{align*}
\end{conjecture}

We also give a related generalisation of \cite[Conjecture 4.13]{GZ18b}.

\begin{conjecture}\label{conj:4k+3-3}
For $n\equiv 3\pmod{4}$ and any positive integer $r$, we have
\begin{align*}
\sum_{k=0}^{rn-1}\frac{(1+q^{4k+1})(q^2;q^4)_k^3}{(1+q)\,(q^4;q^4)_k^3} q^k
&\equiv 0\pmod{\Phi_n(q)^2},
\\
\sum_{k=0}^{rn+(n-1)/2}\frac{(1+q^{4k+1})(q^2;q^4)_k^3}{(1+q)\,(q^4;q^4)_k^3} q^k
&\equiv 0\pmod{\Phi_n(q)^2}.
\end{align*}
\end{conjecture}

Note that, although similar congruences with a parameter $a$ modulo $(1-aq^n)\*(a-\nobreak q^n)$ can be deduced, we cannot
take the limit as $a\to 1$  to accomplish the proof of Conjectures \ref{conj:4k+3-2} and \ref{conj:4k+3-3} this time.
Using the $q$-Lucas theorem, we can show that all the congruences in Conjectures \ref{conj:4k+3-2} and \ref{conj:4k+3-3} are true modulo $\Phi_n(q)$.
Moreover, the following similar congruence in \cite[Theorem 4.14]{GZ18b},
\begin{equation*}
\sum_{k=0}^{(n-1)/2}\frac{(q;q^2)_k ^2}{(q^2;q^2)_k (q^4;q^4)_k}\,q^{2k}\equiv 0\pmod{\Phi_n(q)^2}
\quad\text{for}\; n\equiv 3\pmod{4},
\end{equation*}
does not have such a generalisation.
For this reason, we believe that Conjectures \ref{conj:4k+3-2} and \ref{conj:4k+3-3} are not easy to prove.

\medskip
The $q$-extension \eqref{eq-HqB} of Bauer's formula \eqref{eq-HB} is not unique. In \cite{Guo18b,GZ18a} we mention
a `more classical' version
\begin{equation}
\sum_{k=0}^{\infty}\frac{1-q^{4k+1}}{1-q}\,\frac{(q;q^2)_k^3}{(q^2;q^2)_k^3}\,(-q)^{k^2}
=\frac{(q;q^2)_\infty(q^3;q^2)_\infty}{(q;q^2)_\infty^2}.
\label{eq-HqB2}
\end{equation}
In spite of \eqref{eq-HqB2} (in fact, its parametric modification) being suitable for proving the congruences \eqref{eq-cong0}
on the basis of our method from \cite{GZ18b}, dropping off the factor $(1-q^{4k+1})/(1-q)$ here does not lead to a `suitable' $q$-analogue of \eqref{eq-H2}.
Namely, a numerical check suggests no congruences for the truncated sums of the resulting series.

\medskip
Finally, we notice that the sequence $a_1(p)$ in \eqref{eq-a1} is ultimately linked with a remarkable classics,
the two-square theorem due to Fermat and Gauss (see \cite[Chap.~4]{AZ18}, \cite{CELV99} and also
the unpublished portion of Ramanujan's paper \cite{Ra15}, reproduced in \cite[Chap.~10]{AB12} and relating the two-square
generating function to the Dirichlet $L$-function in~\eqref{eq-L}).
However, no reasonable $q$-analogues of this result have been recorded yet.


\end{document}